\documentclass{amsart}

\usepackage{amsmath}
\usepackage{graphicx}
\usepackage{enumitem}
\usepackage{amssymb}
\usepackage{xcolor}
\usepackage{url}

\newcommand{\veps}{\varepsilon} 
\newcommand{\bbE}{\mathbf{E}}
\newcommand{\Var}{\mathbf{Var}}

\theoremstyle{plain} 
\newtheorem{theorem}{Theorem}[section]
\newtheorem{corollary}[theorem]{Corollary}

\newtheorem{lemma}[theorem]{Lemma}

\theoremstyle{definition}
\newtheorem{definition}[theorem]{Definition}
\newtheorem{example}[theorem]{Example}
\newtheorem{remark}[theorem]{Remark}

\begin{document}

\title{A central limit theorem for a card shuffling problem}



\author{ Shane~Chern}
\address{Department of Mathematics and Statistics, Dalhousie University, Halifax, NS, B3H 4R2, Canada.} \email{chenxiaohang92@gmail.com}
\author{Lin~Jiu$^1$}\footnote{corresponding author}
\address{Zu Chongzhi Center for Mathematics and Computational Sciences, Duke Kunshan University, Kunshan, Jiangsu Province, 215316, China. } 
\email{lin.jiu@dukekunshan.edu.cn}
 \author{Italo~Simonelli}
\address{Zu Chongzhi Center for Mathematics and Computational Sciences, Duke Kunshan University, Kunshan, Jiangsu Province, 215316, China.} 
\email{italo.simonelli@duke.edu}









\begin{abstract}

Given a positive integer $n$, consider a random permutation $\tau$ of the set $\{1,2,\ldots, n\}$. In $\tau$, we look for sequences of consecutive integers that appear in adjacent positions: a maximal such a sequence is called a block. Each block in $\tau$ is merged, and after all the merges, the elements of this new set are relabeled from $1$ to the current number of elements. We continue to randomly permute and merge this new set until only one integer is left. In this paper, we investigate the asymptotic behavior of $X_n$, the number of permutations needed for this process to end. In particular, we find an explicit asymptotic expression for each of $\bbE[X_n]$ and $\Var [X_n]$ as well as for every higher central moment, and show that $X_n$ satisfies a central limit theorem.

\end{abstract}

\maketitle

\section{Introduction}

Given a positive integer $n$, consider a random permutation $\tau$ of the set $[n]:=\{1,2,\ldots, n\}$. In $\tau$, we look for sequences of consecutive integers that appear in adjacent positions, and a maximal such a sequence is called a \emph{block}. Each block in $\tau$ is merged into its first integer, and after all the merges the elements of this new set are relabeled from $1$ to the current number of elements. For example, if $n=10$, and $\tau = (1,7,5,6,8, 10, 9, 2, 3, 4)$, then the blocks are $(1)$, $(7)$, $(5,6)$, $(8)$, $(10)$, $(9)$, and $(2, 3, 4)$.  The merging gives $(1, 7, 5, 8, 10, 9, 2)$ and the relabeling will result in $(1,4,3,5,7,6,2)$, a permutation of the set $[7]= \{1, 2, 3, 4, 5, 6, 7\}$. We randomly permute this new set, and continue this merging and permuting until only one integer is left. The quantity of interest is $X_n$, the number of permutations needed for the process to end. 

This problem was introduced by Rao et al.~\cite{Rao} to model the number of times that catalysts must be added to $n$  molecules to bond into a single lump. The molecules have a given hierarchical order which led to the above mathematical formulation of the process. These authors proved that, for the \emph{mean value} of $X_n$, i.e., $\mu_n:=\bbE[X_n]$, the inequality   $n \leq \bbE[X_n] \leq n + \sqrt{n}$ holds. Hence, 
\begin{align}\label{asymp-mean}
    \bbE[X_n]\sim n.
\end{align}
However, they did not provide any explicit expression for $\bbE[X_n]$ or estimate for $\Var[X_n]$. Instead, they simulated the distribution of $X_n$ for $n=2,\ldots, 100$, and their calculations hinted that the asymptotic distribution of $X_n$ would tend to be normal.  

In this paper, we provide a solution to this problem in the affirmative, as follows.

\begin{theorem}\label{th:main}
Let $Z$ denote a standard normal random variable, i.e., $Z \sim \mathcal{N}(0,1)$. Then as $n\to +\infty$,
\begin{align}\label{eq:main}
    \frac{X_n -n}{\sqrt{n}} \stackrel{w}{\longrightarrow} Z.
\end{align}
\end{theorem}

The proof of the above central limit theorem follows directly from the asymptotic estimate of each central moment of $X_n$. To present these asymptotic relations, we begin with some notation. First, the conventional \emph{Bachmann--Landau symbols} are to be used throughout this paper: 
\begin{itemize}[leftmargin=*,align=left,itemsep=6pt]
	\item 
    If there exists a constant $C$ such that $|f(n)|\le C g(n)$, we adopt the \emph{big-$O$ notation} and write $f(n)=O\big(g(n)\big)$;

    \item 
    If $\lim_{n\rightarrow+\infty} f(n)/g(n)=0$, we adopt the \emph{small-$o$ notation} and write $f(n)=o\big(g(n)\big)$;

    \item 
    If $\lim_{n\rightarrow+\infty} f(n)/g(n)=1$, we adopt the \emph{asymptotic equivalence symbol} and write $f(n)\sim g(n)$.
\end{itemize}
Also, the \emph{double factorials} are defined by 
\begin{align*}
    n!!:=n\cdot (n-2) \cdot (n-4) \cdots
\end{align*}
Now our main result can be stated as follows.

\begin{theorem}\label{th:c-mom}
	For every $m\ge 2$, we have, as $n\to +\infty$,
	\begin{align}\label{eq:mom-asymp}
		\bbE\big[(X_n-\mu_n)^m\big] = \begin{cases}
			(2M-1)!!\cdot n^M + O(n^{M-1}\log n), & \text{if $m=2M$},\\
			\tfrac{2}{3}M(2M+1)!!\cdot n^M + O(n^{M-1}\log n), & \text{if $m=2M+1$},
		\end{cases}
	\end{align}
	wherein the asymptotic relation depends only on $m$.
\end{theorem}

\begin{proof}[Proof of Theorem \ref{th:main} from Theorem \ref{th:c-mom}]
    Directly, \eqref{eq:mom-asymp} indicates that 
        \begin{align}\label{eq:Var-asymp}
		\Var[X_n] = \bbE\big[(X-\mu_n)^2\big] &\sim n,
	\end{align}
    Therefore, for $Z_n := (X_n - \mu_n)/\sqrt{\Var[X_n]}$, we have $\bbE[Z_n]=0$, $\Var[Z_n]=1$, and more importantly, as $n\to +\infty$, 
    \begin{align*}
        \bbE\big[Z_n^m\big] =  \frac{\bbE\big[(X_n-\mu_n)^m\big]}{\Var[X_n]^{m/2}} \to 
        \begin{cases}
	   (m-1)!!, & \text{if $m$ is even},\\
	   0, & \text{if $m$ is odd},
        \end{cases}
    \end{align*}
    matching with the moments of the standard normal distribution $Z$. By recourse to \emph{Chebyshev's method of moments} \cite[p.~390, Theorem 30.2]{Bil1995}, the weak convergence of $Z_n \Rightarrow Z$ is clear. 
    Lastly, the asymptotic relations \eqref{asymp-mean} and \eqref{eq:Var-asymp} assert the validity of Theorem \ref{th:main}.
\end{proof}

As now the main task is to prove Theorem \ref{th:c-mom}, we shall organize the paper in the following way. In Section \ref{sec:A(n,k)}, we review the shuffling model in \cite{Rao} and derive a handful of properties of its associated probability sequence. Those properties will be utilized in Section \ref{sec:asymp-generic} for the asymptotic analysis in regard to a family of generic recurrences. 
In Section \ref{sec:mean-value}, a more precise asymptotic estimate of $\mu_n$ is obtained; and in particular, Subsection \ref{sec:Bell} is devoted to a surprising connection between the mean values and the Bell numbers. Finally, in Section~\ref{sec:var-3moment}, we elaborate on Theorem \ref{th:c-mom} with a delicate estimate of the error terms, and prove this strengthening, as stated in Theorem \ref{th:c-mom-new}, by an inductive argument.

\subsection*{Convention}

\begin{itemize}[leftmargin=*,align=left,itemsep=6pt]
    \item Throughout the rest of the paper, all limits and asymptotic relations are to be interpreted with respect to $n$, as $n\to +\infty$.
    \item In this work we may occasionally encounter terms of the form $0^0$ (e.g., the summand $(n-k)^\ell$ in \eqref{eq:A-(n-k)u} when $k=n$ and $\ell=0$). In such scenarios, we simply treat $0^0$ as $1$.
\end{itemize}

\section{The shuffling model and probability sequence}\label{sec:A(n,k)}
\subsection{The shuflling model}
To formulate the model in this work, we recall that by denoting $Y_n$ the number of blocks in a random permutation of $[n]$, it is well-known (cf.~\cite[p.~615, eq.~(1)]{Rao}) that for $k=1, 2, \ldots, n$, the following probability evaluation holds:
$$\mathbf{P}(Y_n = k) = \frac{A(n,k)}{n!}, $$ 
where the sequence $A(n,k)$ is defined for $1\leq k \leq n$ by
$$ A(n,k) := \binom{n-1}{k-1} A(k-1),$$
with the sequence $A(n)$ being given by 
\begin{align*}
    A(n) := \frac{(n+2)!}{n+1}\sum_{i=0}^{n+2}\frac{(-1)^i}{i!}.
\end{align*}
Hence, summing $\mathbf{P}(Y_n = k)$ over $k$ gives
\begin{equation}\label{sum=1}
\sum_{k=1}^n \frac{A(n,k)}{n!} =1.
\end{equation}
Meanwhile, for $1 \leq k \leq n$, let $\widetilde{X}_k$ be independent random variables such that $\widetilde{X}_k \stackrel{d}{=}X_k$. That is, $\widetilde{X}_k$ is the number of permutations needed for the shuffling process to end with the initial state being a permutation of the set $[k]$. A crucial relation shown by Rao et al.~is \cite[p.~616, eq.~(4)]{Rao}:
\begin{align*}
    X_n = \begin{cases}
        1, & \text{with probability $\mathbf{P}(Y_n = 1)$},\\
        1+\widetilde{X}_k, & \text{with probability $\mathbf{P}(Y_n = k)$ for $2\le k\le n$}.
    \end{cases}
\end{align*}
With this, it was shown in \cite{Rao} that the mean value of $X_n$, that is, $\mu_n=\bbE[X_n]$, satisfies the recurrence relation:
\begin{align*} \mu_n &  =   \bbE\big[\bbE[X_n\,|\,Y_n]\big]\\
&=   \mathbf{P}(Y_n=1) + \sum_{k=2}^n \mathbf{P}(Y_n=k) \bbE\big[1+ \widetilde{X}_k\big]\\
& = \frac{A(n,1)}{n!}+\sum_{k=2}^{n} \frac{A(n,k)}{n!} (1+\mu_k)\\
& = 1 + \sum_{k=2}^{n} \frac{A(n,k)}{n!} \mu_k,
\end{align*}
where we have applied \eqref{sum=1}. Now to facilitate our analysis, we adopt the convention that $X_{1}=0$, a definite constant, so that the mean value $\mu_{1}=\bbE[X_{1}]=0$, as there is no more shuffling needed. Then the above recurrence for $\mu_n$ becomes
\begin{equation} \label{rec-1}
\mu_n = 1 + \sum_{k=1}^{n} \frac{A(n,k)}{n!} \mu_k,
\end{equation}
with initial values $\mu_1=0$ and $\mu_2=2$.

In general, if we assume that $p(x)$ is an arbitrary polynomial in $x$, then
\begin{align}\label{eq:E-poly}
    \bbE\big[p(X_n)\big] = \sum_{k=1}^n \mathbf{P}(Y_n=k)\bbE\big[p(1+X_k)\big] = \sum_{k=1}^n \frac{A(n,k)}{n!}\cdot \bbE\big[p(1+X_k)\big].
\end{align}
This relation will serve as a key in our subsequent evaluation of the central moments of $X_n$.




Here, we offer a historical remark on the sequence $A(n)$, which is registered as sequence \textsf{A000255} in the On-Line Encyclopedia of Integer Sequences \cite{OEIS}. This sequence was first introduced by Euler \cite[p.~235, Sect.~152]{Eul1782} in his study of Latin squares, and Euler interpreted that this sequence enumerates ``fixed-point-free permutations beginning with $2$.'' For modern studies of the numbers $A(n)$ and $A(n,k)$, see, for example, Kreweras \cite{Kreweras} and Myers \cite{Myers}. 

\subsection{Properties of the probability sequence}

In this subsection, we take a deeper look at the double sequence $A(n,k)$. To facilitate our investigation of $A(n,k)$, an explicit formula for its generating function becomes much needed. We start by noting that the following exponential generating function of the sequence $A(n)$ was already derived in \cite{Kreweras}. However, for the sake of completeness, we sketch a proof.

\begin{lemma}
    We have
    \begin{equation} \label{eq:A-egf}
        \sum_{n=0}^{\infty} \frac{A(n)}{n!}x^n = \frac{e^{-x}}{(1-x)^2}.
    \end{equation}
\end{lemma}

\begin{proof}
    It is known from \cite[p.~178, eq.~(5.9)]{Cha2002} that $A(n)$ satisfies the recurrence relation:
    \begin{align}\label{eq:A-rec}
        A(n) = n A(n-1)+(n-1) A(n-2),
    \end{align}
    with initial values $A(0) = A(1) =1$. Let 
	\begin{align*}
		F(x):= \sum_{n=0}^\infty A(n)\frac{x^n}{n!}.
	\end{align*}
	Now we multiply by $\frac{x^n}{n!}$ on both sides of \eqref{eq:A-rec} and then sum over $n\ge 2$. Thus,
	\begin{align*}
		\sum_{n=2}^\infty A(n) \frac{x^n}{n!} = \sum_{n=2}^\infty A(n-1) \frac{x^n}{(n-1)!} + \sum_{n=2}^\infty A(n-2) \frac{x^n}{(n-2)!}\cdot\frac{1}{n},
	\end{align*}
	namely,
	\begin{align*}
		F(x)-\big(1+x\big) = x \big(F(x)-1\big) + \sum_{n=2}^\infty A(n-2) \frac{x^n}{(n-2)!}\cdot \frac{1}{n}.
	\end{align*}
	Differentiating both sides with respect to $x$ then gives
	\begin{align*}
		F'(x) - 1 = F(x)-1 + xF'(x) + xF(x).
	\end{align*}
	Solving the above PDE under the boundary condition that $F(0)=1$, we arrive at the claimed relation.
\end{proof}

Now we are ready to provide a bivariate generating function identity for $A(n,k)$.

\begin{theorem}
	We have
	\begin{align}\label{eq:A-bivariate-gf}
		\sum_{n=1}^\infty \left(\sum_{k=1}^n \frac{A(n,k)}{n!}z^k\right) nx^n = \frac{xze^{x(1-z)}}{(1-xz)^2}.
	\end{align}
\end{theorem}

\begin{proof}
	We make use of \eqref{eq:A-egf} to get
	\begin{align*}
		\frac{xze^{x(1-z)}}{(1-xz)^2} &= e^x\cdot \frac{xze^{-xz}}{(1-xz)^2}\\
		&=\left(\sum_{j=0}^\infty \frac{x^j}{j!}\right) \left(\sum_{\ell=0}^\infty \frac{A(\ell)}{\ell!} x^{\ell+1}z^{\ell+1}\right)\\
        &= \sum_{n=1}^\infty \left(\sum_{k=1}^n \frac{1}{(n-k)!} \frac{A(k-1)}{(k-1)!} z^k\right) x^n\\
		&= \sum_{n=1}^\infty \left(\sum_{k=1}^n \binom{n-1}{k-1} \frac{A(k-1)}{n!} z^k\right) nx^n,
	\end{align*}
    wherein the third equality follows by performing the Cauchy product with respect to $x$. Finally, the claimed relation holds by recalling that $A(n,k) = \binom{n-1}{k-1}A(k-1)$.
\end{proof}

In \eqref{eq:A-bivariate-gf}, by expanding $e^{x(1-z)}$ and $xz/(1-xz)^2$ as a series in $x$, respectively, the following relation is clear.

\begin{corollary}\label{coro:A-z-gf}
	For $n\ge 1$,
	\begin{align}\label{eq:A-z-gf}
		\sum_{k=1}^n \frac{A(n,k)}{n!}z^k = \sum_{m=0}^{n-1}\frac{n-m}{n\cdot m!}z^{n-m}(1-z)^m.
	\end{align}
\end{corollary}

\begin{remark}
    Taking $z=1$ in \eqref{eq:A-z-gf} gives an alternative proof of \eqref{sum=1}.
\end{remark}

We may further deduce from \eqref{eq:A-z-gf} the following evaluation:

\begin{theorem}\label{th:A-(n-k)u}
	Let $\ell$ be a nonnegative integer. Then for $n\ge \ell$,
	\begin{align}\label{eq:A-(n-k)u}
		\sum_{k=1}^n \frac{A(n,k)}{n!} (n-k)^\ell = B_\ell - (B_{\ell+1}-B_\ell)\cdot \frac{1}{n},
	\end{align}
	where $B_\ell$ is the $\ell$-th Bell number defined by the exponential generating function
	\begin{align*}
		\sum_{\ell= 0}^\infty B_\ell \frac{x^\ell}{\ell!} := e^{e^x-1}.
	\end{align*}
\end{theorem}


\begin{proof}
	The $\ell=0$ case is exactly \eqref{sum=1}. Throughout, we assume that $\ell\ge 1$. In light of \eqref{eq:A-z-gf}, it is plain that
	\begin{align*}
		\sum_{k=1}^n \frac{A(n,k)}{n!}z^{n-k} = \sum_{m=0}^{n-1}\frac{n-m}{n\cdot m!}(z-1)^m.
	\end{align*}
	Thus,
	\begin{align*}
		\sum_{k=1}^n \frac{A(n,k)}{n!} (n-k)^\ell = \sum_{m=0}^{n-1}\frac{n-m}{n\cdot m!} \left[\left(z\frac{d}{dz}\right)^\ell (z-1)^m\right]_{z=1},
	\end{align*}
	where $(z\frac{d}{dz})^\ell$ means applying $\ell$ times the operator $z\frac{d}{dz}$. Note that when $m>\ell$, we always have a factor $(z-1)$ in $(z\frac{d}{dz})^\ell (z-1)^m$, thereby implying that this quantity vanishes when taking $z=1$. Hence,
	\begin{align*}
		\sum_{k=1}^n \frac{A(n,k)}{n!} (n-k)^\ell &= \sum_{m=0}^{\min\{n-1,\ell\}}\frac{n-m}{n\cdot m!} \left[\left(z\frac{d}{dz}\right)^\ell (z-1)^m\right]_{z=1}\\
		&= \sum_{m=0}^{\min\{n-1,\ell\}}\frac{n-m}{n\cdot m!} \left[\left(z\frac{d}{dz}\right)^\ell \sum_{i=0}^m \binom{m}{i}(-1)^{m-i}z^i\right]_{z=1}\\
		&= \sum_{m=0}^{\min\{n-1,\ell\}}\frac{n-m}{n\cdot m!} \sum_{i=0}^m \binom{m}{i}(-1)^{m-i} i^\ell\\
		&= \sum_{m=0}^{\min\{n-1,\ell\}}\frac{n-m}{n} {\ell \brace m},
	\end{align*}
	where ${\ell\brace m}$ are the \emph{Stirling numbers of the second kind}. Therefore, for $n\ge \ell$,
	\begin{align}\label{eq:A-(n-k)u-Stirling}
		\sum_{k=1}^n \frac{A(n,k)}{n!} (n-k)^\ell = \sum_{m=0}^{\ell}\frac{n-m}{n} {\ell\brace m}.
	\end{align}
	Now it is known from \emph{Dobi\'nski's formula} \cite[p.~210, eq.~(4a)]{Com1974}, with ${\ell\brace 0} = 0$ for positive integers $\ell$ in mind, that 
	\begin{align}\label{eq:Dob}
		\sum_{m=0}^{\ell} {\ell\brace m} = B_\ell.
	\end{align}
	Meanwhile, we recall the following recurrence for the Stirling numbers of the second kind \cite[p.~20, eq.~(1.34)]{Wil2006}:
	\begin{align*}
		m{\ell\brace m} = {\ell+1\brace m}-{\ell\brace m-1},
	\end{align*}
	with the convention that ${\ell\brace m} = 0$ whenever $m<0$. Hence, by also recalling that ${\ell+1\brace \ell+1} ={\ell\brace \ell} = 1$, we have
	\begin{align}\label{eq:Dob*m}
		\sum_{m=0}^{\ell} m{\ell\brace m} = \sum_{m=0}^{\ell+1} {\ell+1\brace m} - \sum_{m=0}^{\ell} {\ell\brace m} = B_{\ell+1}-B_{\ell},
	\end{align}
	where Dobi\'nski's formula is applied again. Finally, the desired relation \eqref{eq:A-(n-k)u} follows by substituting \eqref{eq:Dob} and \eqref{eq:Dob*m} into \eqref{eq:A-(n-k)u-Stirling}.
\end{proof}

The identity \eqref{eq:A-(n-k)u} gives us the following estimate:

\begin{corollary}\label{coro:Q-s-bound}
	Let $\ell$ be a nonnegative integer and let $s$ be an arbitrary real number. Then as $n\to +\infty$,
	\begin{align}\label{eq:Q-s-bound}
		\sum_{k=1}^n \frac{A(n,k)}{n!}(n-k)^\ell\, k^s = O(n^{s}).
	\end{align}
	Here the asymptotic relation depends only on $\ell$ and $s$.
\end{corollary}

In what follows, an important recipe that will be repeatedly utilized is the \emph{Abel summation formula} \cite[p.~3, Theorem 0.1]{Ten2015}: 

\begin{lemma}[Abel summation formula]
	Let $\{u_n\}_{n\ge 1}$ and $\{v_n\}_{n\ge 1}$ be sequences of complex numbers. Then for any $N\ge 1$,
	\begin{align*}
		\sum_{n=1}^N u_n v_n = U(N) v_{N+1} + \sum_{n=1}^N U(n) \big(v_n-v_{n+1}\big),
	\end{align*}
	where $U(n):=\sum_{k=1}^n u_k$.
\end{lemma}

Now we are ready to prove Corollary \ref{coro:Q-s-bound}.

\begin{proof}[Proof of Corollary \ref{coro:Q-s-bound}]
	Note that the summand with $k=n$ in \eqref{eq:Q-s-bound} is $0$ when $\ell>0$, and $O(n^{s})$ when $\ell=0$ since $0\le \frac{A(n,n)}{n!}\le 1$ in view of \eqref{sum=1}. Hence, it is sufficient to show that
	\begin{align}\label{eq:Q-s-bound-new}
		\sum_{k=1}^{n-1} \frac{A(n,k)}{n!}(n-k)^\ell\, k^s = O(n^{s}).
	\end{align}
	
	The cases where $s\ge 0$ are simpler as we directly have
	\begin{align*}
		\sum_{k=1}^{n-1} \frac{A(n,k)}{n!}(n-k)^\ell\, k^s \le n^s\cdot \sum_{k=1}^{n-1} \frac{A(n,k)}{n!}(n-k)^\ell = O(n^s),
	\end{align*}
	where \eqref{eq:A-(n-k)u} has been applied to derive that the summation in the middle part of the above is $O(1)$.
	
	In what follows, we always assume that $s<0$. We begin with the cases where $s$ is an integer. Define the following partial sums with $1\le t\le n$:
	\begin{align*}
		Q_n(t)=Q_n^{(\ell,s)}(t):= \sum_{k=1}^t \frac{A(n,k)}{n!}(n-k)^{\ell-s}.
	\end{align*}
	Recalling that \eqref{eq:A-(n-k)u} advocates that $Q_n(n)$ converges as $n$ goes to infinity, we may find a constant $K$, depending on $\ell$ and $s$, such that for every $n\ge 1$,
	\begin{align*}
		0\le Q_n(1)\le Q_n(2)\le \cdots \le Q_n(n) \le K.
	\end{align*}
	Now we apply the Abel summation formula to obtain
	\begin{align*}
		&\quad\; \sum_{k=1}^{n-1} \frac{A(n,k)}{n!}(n-k)^\ell\, k^s\\
		&= \sum_{k=1}^{n-1} \frac{A(n,k)}{n!}(n-k)^{\ell-s}\cdot k^s (n-k)^s\\
		&= Q_n(n-1)\cdot (n-1)^s + \sum_{t=1}^{n-2} Q_n(t)\cdot \big(t^s (n-t)^s-(t+1)^s (n-t-1)^s\big).
	\end{align*}
	Note that the sequence $\{t (n-t)\}_{1\le t\le (n-1)}$ is unimodal. That is, if we let $T=\left\lfloor\frac{n}{2}\right\rfloor$ where the \emph{floor function} $\lfloor x\rfloor$ denotes the greatest integer not exceeding $x$, then
	\begin{align*}
		1\cdot (n-1)\le \cdots \le T\cdot (n-T)\ge \cdots \ge (n-1)\cdot \big(n-(n-1)\big).
	\end{align*}
	We then reformulate the above summation as
	\begin{align*}
		\sum_{k=1}^{n-1} \frac{A(n,k)}{n!}(n-k)^\ell\cdot k^s &= Q_n(n-1)\cdot (n-1)^s\\
		&\quad + \sum_{t=1}^{T-1} Q_n(t)\cdot \big(t^s (n-t)^s-(t+1)^s (n-t-1)^s\big)\\
		&\quad+ \sum_{t=T}^{n-2} Q_n(t)\cdot \big(t^s (n-t)^s-(t+1)^s (n-t-1)^s\big).
	\end{align*}
	First, since $Q_n(n-1)\le K$, we have the estimate 
	\begin{align*}
		Q_n(n-1)\cdot (n-1)^s = O(n^{s}).
	\end{align*}
	Next, bearing in mind that $s<0$, we have
	\begin{align*}
		t^s (n-t)^s-(t+1)^s (n-t-1)^s
		\begin{cases}
			\ge 0, & \text{if $1\le t\le T-1$},\\
			\le 0, & \text{if $T\le t\le n-2$}.
		\end{cases}
	\end{align*}
	Thus,
	\begin{align*}
		0\le \sum_{t=1}^{T-1} Q_n(t)\cdot \big(t^s (n-t)^s-(t+1)^s (n-t-1)^s\big)\le K\cdot \big((n-1)^s - T^s (n-T)^s\big),
	\end{align*}
	and
	\begin{align*}
		0\ge \sum_{t=T}^{n-2} Q_n(t)\cdot \big(t^s (n-t)^s-(t+1)^s (n-t-1)^s\big)\ge K\cdot \big(T^s (n-T)^s - (n-1)^s\big),
	\end{align*}
	while it is plain that
	\begin{align*}
		K\cdot \big((n-1)^s - T^s (n-T)^s\big) = O(n^{s}).
	\end{align*}
	Hence, \eqref{eq:Q-s-bound-new} is valid for integral $s<0$.
	
	If $s<0$ is not an integer, then we write $s_0=\lfloor s\rfloor$ so that $s-s_0\ge 0$. It follows that
	\begin{align*}
		\sum_{k=1}^{n-1} \frac{A(n,k)}{n!}(n-k)^\ell\, k^s &= \sum_{k=1}^{n-1} \frac{A(n,k)}{n!}(n-k)^\ell\, k^{s_0}\cdot k^{s-s_0}\\
		&\le n^{s-s_0} \cdot \sum_{k=1}^{n-1} \frac{A(n,k)}{n!}(n-k)^\ell\, k^{s_0}.
	\end{align*}
	We have shown that the last summation is $O(n^{s_0})$. The required estimate immediately follows.
\end{proof}

From now on, for any given $n\ge 1$, we define the partial sums $S_n(t)$ with $1\le t\le n$:
\begin{align*}
	S_n(t):= \sum_{k=1}^t \frac{A(n,k)}{n!}.
\end{align*}
In light of \eqref{sum=1}, we particularly have
\begin{equation} \label{eq:Snn=1}
S_n(n) =1.
\end{equation}
Meanwhile, for $S_n(n-1)$, we have an asymptotic estimate as follows:

\begin{lemma}
	As $n\to +\infty$,
	\begin{align}\label{eq:Sn(n-1)-asymp}
	S_n(n-1)\sim 1-e^{-1}.
	\end{align}
\end{lemma}

\begin{proof}
	We invoke the following standard result \cite[p.~178, Remark 5.4 and p.~184, Exercise 5.7.1]{Cha2002}:
	\begin{align}\label{eq:A(n)-exact}
		A(n) = \left\lfloor \frac{(n+2)\cdot n!}{e}+\frac{1}{2}\right\rfloor.
	\end{align}
	Recalling from \eqref{sum=1} that $S_n(n-1) = 1 - \frac{A(n-1)}{n!}$, we immediately arrive at the required relation.
\end{proof}


\begin{remark}
	By performing a more delicate computation on the exact formula \eqref{eq:A(n)-exact} for $A(n)$, we further have explicit bounds: \textit{For $n\ge 200$},
	\begin{align}\label{eq:Sn(n-1)-bound}
		0.63 < S_n(n-1)< 0.64.
	\end{align}
\end{remark}

We end this section by deriving some useful identities for $S_n(t)$.

\begin{lemma}
	For $n\ge 1$,
	\begin{align}\label{eq:Sn-gf}
		\sum_{t=1}^n S_n(t)z^t = z^n + \sum_{m=1}^{n-1}\frac{n-m}{n\cdot m!}z^{n-m}(1-z)^{m-1}.
	\end{align}
\end{lemma}

\begin{proof}
	We have, by interchanging the order of summations, that
	\begin{align*}
		\sum_{t=1}^n S_n(t)z^t = \sum_{t=1}^n \sum_{k=1}^t \frac{A(n,k)}{n!}z^t
		= \sum_{k=1}^n \frac{A(n,k)}{n!} \sum_{t=k}^n z^t
		= \sum_{k=1}^n \frac{A(n,k)}{n!} \frac{z^k-z^{n+1}}{1-z}.
	\end{align*}
	Invoking Corollary \ref{coro:A-z-gf} gives the desired result.
\end{proof}

%

We apply $\ell$ times the operator $z\frac{d}{dz}$ to \eqref{eq:Sn-gf} and then take $z=1$ to obtain the following estimate:

\begin{corollary}\label{coro:S_n-t-power-sums-l}
	Let $\ell$ be a nonnegative integer. Then as $n\to +\infty$,
	\begin{align}\label{eq:sum-S-1}
		\sum_{t=1}^n S_n(t)\cdot t^\ell = 2n^{\ell} +O(n^{\ell-1}),
	\end{align}
	where the asymptotic relation depends only on $\ell$.
\end{corollary}

Meanwhile, by applying the operator $[\int z^{-1}\bullet dz]$ zero times, once and twice to \eqref{eq:Sn-gf} and then let $z=1$, respectively, we have the following relations:

\begin{corollary}\label{coro:S_n-t-power-sums}
	For $n\ge 1$,
	\begin{align}
		\sum_{t=1}^n S_n(t) &= 2-\frac{1}{n},\label{eq:sum-S-0}\\
		\sum_{t=1}^n S_n(t)\cdot \frac{1}{t} &= \frac{1}{n} + \sum_{m=1}^{n-1}\frac{(n-m)!}{m\cdot n!},\label{eq:sum-S--1}\\
		\sum_{t=1}^n S_n(t)\cdot \frac{1}{t^2} &= \frac{1}{n^2} + \sum_{m=1}^{n-1}\frac{(n-m-1)!}{n\cdot n!} + \sum_{m=1}^{n-1}\frac{(n-m)!}{m\cdot n!}\big(\mathcal{H}_n - \mathcal{H}_{n-m}\big),\label{eq:sum-S--2}
	\end{align}
    where $\mathcal{H}_n=\sum_{k=1}^n \frac{1}{k}$ is the $n$-th harmonic number. 
\end{corollary}

\section{Asymptotics for a generic recurrence}\label{sec:asymp-generic}

Recall that the recurrence \eqref{rec-1} for the mean values $\mu_n$ can be reformulated as
\begin{align}\label{eq:mu-original}
	\left(1-\frac{A(n,n)}{n!}\right)\mu_n = 1+\sum_{k=1}^{n-1}\frac{A(n,k)}{n!}\mu_k,
\end{align}
However, as hinted by \eqref{eq:E-poly}, the recurrence relations \eqref{eq:mth-M} for the central moments of $X_n$ have a more complicated form in which the $1$ on the right-hand side of \eqref{eq:mu-original} is replaced with much more sophisticated terms. This motivates the investigation of the asymptotics for sequences satisfying a generic recurrence.

\begin{definition}
	Let $\{\lambda_n\}_{n\ge 1}$ be a complex sequence such that $\lambda_n\sim Mn^L$ as $n\to +\infty$, wherein $L$ is a fixed nonnegative integer, and $M$ is a fixed complex number, which, in addition, is nonzero when $L$ is nonzero. Write
	\begin{align*}
		\delta_n:=\lambda_n-Mn^L.
	\end{align*}
	We define a complex sequence $\{\xi_n\}_{n\ge 1}$ with given initial values $\xi_1,\ldots,\xi_{n_0}$ for a certain $n_0\ge 2$ by the recurrence
	\begin{align}\label{eq:rec:xi}
		\left(1-\frac{A(n,n)}{n!}\right)\xi_n = \lambda_n+\sum_{k=1}^{n-1}\frac{A(n,k)}{n!}\xi_k,
	\end{align}
	for every $n>n_0$.
\end{definition}

We are interested in the growth rate of $\xi_n$.

\begin{theorem}\label{th:xi-asymp}
	As $n\to+\infty$,
	\begin{align}\label{eq:xi-asymp}
		\xi_n\sim \frac{M}{L+1}\,n^{L+1},
	\end{align}
	where the asymptotic relation depends only on $L$ and $M$. More precisely, letting
	\begin{align}\label{eq:eta}
		\eta_n:=\xi_n-\frac{M}{L+1}\,n^{L+1},
	\end{align}
	there exists a positive constant $C$, depending only on $L$ and $M$, such that for all $n\ge 1$,
	\begin{align}\label{eq:eta-bound}
		\big|\eta_n\big| < C \sum_{j=1}^n  \left(\big|\delta_j\big|+j^{L-1}\right).
	\end{align}
	
\end{theorem}

\begin{proof}
	It is clear that \eqref{eq:eta-bound} implies \eqref{eq:xi-asymp}. Throughout, we assume that $n$ is such that $n>n_0$. In \eqref{eq:rec:xi}, we first apply \eqref{sum=1} to rewrite the multiplier on the left-hand side as
	\begin{align*}
		1-\frac{A(n,n)}{n!} = S_n(n-1),
	\end{align*}
	and then replace all $\xi_k$ on the right-hand side according to \eqref{eq:eta}. Thus,
	\begin{align}\label{eq:xi-Sigma}
		S_n(n-1)\cdot \xi_n = \lambda_n+ \Sigma_1+\Sigma_2,
	\end{align}
	where
	\begin{align*}
		\Sigma_1 &:= \sum_{k=1}^{n-1} \frac{A(n,k)}{n!}\frac{M}{L+1}\,k^{L+1},\\
		\Sigma_2 &:= \sum_{k=1}^{n-1} \frac{A(n,k)}{n!}\eta_k.
	\end{align*}
	
	Now we evaluate $\Sigma_1$ by means of the Abel summation formula:
	\begin{align*}
		\Sigma_1 &= \sum_{k=1}^{n-1} \frac{A(n,k)}{n!}\frac{M}{L+1}\,k^{L+1}\\
		&= S_n(n-1) \cdot \frac{M}{L+1}\,n^{L+1} + \sum_{k=1}^{n-1} S_n(k)\cdot \frac{M}{L+1} \big(k^{L+1}-(k+1)^{L+1}\big)\\
		&= S_n(n-1) \cdot \frac{M}{L+1}\,n^{L+1} - \sum_{k=1}^{n-1} S_n(k)\cdot \big(M k^L +O(k^{L-1})\big),
	\end{align*}
	where the big-$O$ term further vanishes in the case where $L=0$. Recalling that $S_n(n)=1$ from \eqref{eq:Snn=1}, we then apply \eqref{eq:sum-S-1} and obtain that
	\begin{align*}
		\sum_{k=1}^{n-1} S_n(k)\cdot \big(M k^L +O(k^{L-1})\big) = Mn^L + O(n^{L-1}).
	\end{align*}
	Thus,
	\begin{align}\label{eq:Sigma1}
		\Sigma_1 = S_n(n-1) \cdot \frac{M}{L+1}\,n^{L+1} - M n^L + O(n^{L-1}).
	\end{align}
	
	Substituting \eqref{eq:Sigma1} into \eqref{eq:xi-Sigma} implies that
	\begin{align*}
		S_n(n-1)\cdot \xi_n = \lambda_n+S_n(n-1) \cdot \frac{M}{L+1}\,n^{L+1} - Mn^L + \sum_{k=1}^{n-1} \frac{A(n,k)}{n!}\eta_k + O(n^{L-1}),
	\end{align*}
	which further gives us
	\begin{align}\label{eq:e-new}
		S_n(n-1)\cdot \eta_{n} = \delta_n + \sum_{k=1}^{n-1} \frac{A(n,k)}{n!}\eta_k + O(n^{L-1}),
	\end{align}
	where we have used the facts that $\xi_n-\frac{M}{L+1}\,n^{L+1} = \eta_n$ and that $\lambda_n-Mn^L = \delta_n$.
	
	Now we find a constant $C$ such that
	\begin{enumerate}[label=\textbf{(\alph*)~},leftmargin=*,labelsep=0cm,align=left,itemsep=6pt]
		\item $C\ge 2$;
		
		\item the inequality \eqref{eq:eta-bound} holds for every $n$ with $1\le n\le n_0$;
		
		\item the big-$O$ term in \eqref{eq:e-new} is bounded from above by $\frac{1}{2}C n^{L-1}$, that is,
		\begin{align}\label{eq:xi-simple-error-O}
			\big|O(n^{L-1})\big| < \frac{1}{2}C n^{L-1},
		\end{align}
		for every $n\ge 1$ (this is admissible since the big-$O$ term depends only on $L$ and $M$).
	\end{enumerate}
	
	Herein we prove \eqref{eq:eta-bound} inductively by first bounding for each $k$ such that $1\le k\le n-1$ for a certain $n>n_0$,
	\begin{align*}
		\big|\eta_k\big| < C \sum_{j=1}^{k}  \left(\big|\delta_j\big|+j^{L-1}\right)\le C \sum_{j=1}^{n-1}  \left(\big|\delta_j\big|+j^{L-1}\right).
	\end{align*}
	Thus,
	\begin{align*}
		\left|\sum_{k=1}^{n-1} \frac{A(n,k)}{n!}\eta_k\right| < C \sum_{j=1}^{n-1}  \left(\big|\delta_j\big|+j^{L-1}\right)\cdot \sum_{k=1}^{n-1} \frac{A(n,k)}{n!}
		= C \sum_{j=1}^{n-1} \left(\big|\delta_j\big|+j^{L-1}\right)\cdot S_n(n-1).
	\end{align*}
	It follows from \eqref{eq:e-new}, with \eqref{eq:xi-simple-error-O} recalled, that
	\begin{align*}
		\big|\eta_n\big| < \left(\big|\delta_n\big| + \frac{1}{2}Cn^{L-1}\right) \frac{1}{S_n(n-1)} + C \sum_{j=1}^{n-1}  \left(\big|\delta_j\big|+j^{L-1}\right).
	\end{align*}
	Bearing in mind from \eqref{sum=1},
	\begin{align*}
		\frac{1}{S_n(n-1)} = \frac{1}{1- \frac{A(n,n)}{n!}} \le \frac{1}{1-\frac{A(3,3)}{3!}} = 2,
	\end{align*}
	where we have used the assumption that $n>n_0\ge 2$. Then, it is immediately clear that \eqref{eq:eta-bound} holds for $n$ since we have assumed that $C\ge 2$.
\end{proof}

We may further elaborate on the error term $\eta_n$ when $\delta_n$ is efficiently bounded.

\begin{theorem}\label{th:xi-asymp-error}
	With the notation in Theorem \ref{th:xi-asymp}, we further have
	\begin{enumerate}[label=\textbf{(\roman*)},leftmargin=*,labelsep=0cm,align=left,itemsep=6pt]
		
		\item In the case where $L=0$, if we further require that $\delta_n = O(n^{-1})$ as $n\to +\infty$, then
		\begin{align}\label{eq:eta-diff-bound}
			\big|\eta_n-\eta_{n-1}\big| = O(n^{-1}).
		\end{align}
		
		\item In the case where $L\ge 1$, if we further require that $\delta_n = O(n^{L-1}\log n)$ as $n\to +\infty$, then
		\begin{align}\label{eq:eta-diff-bound-2}
			\big|\eta_n-\eta_{n-1}\big| = O(n^{L-1}\log n).
		\end{align}
		
	\end{enumerate}
	The above asymptotic relations depend only on $L$ and $M$.
\end{theorem}

We first need to apply the Abel summation formula to further reformulate $\Sigma_2$ in \eqref{eq:xi-Sigma}:
\begin{align}\label{eq:Sigma2}
	\Sigma_2 &= \sum_{k=1}^{n-1} \frac{A(n,k)}{n!}\eta_k\notag\\
	&= S_n(n-1) \cdot \eta_n + \sum_{k=1}^{n-1} S_n(k)\cdot  \big(\eta_k-\eta_{k+1}\big)\notag\\
	&= S_n(n-1) \cdot \eta_{n-1} + \sum_{k=1}^{n-2} S_n(k)\cdot  \big(\eta_k-\eta_{k+1}\big).
\end{align}
Substituting \eqref{eq:Sigma2} into \eqref{eq:e-new}, we have
\begin{align}\label{eq:e-diff-new}
	S_n(n-1)\cdot \big(\eta_{n}-\eta_{n-1}\big) = \delta_n + \sum_{k=1}^{n-2} S_n(k)\cdot  \big(\eta_k-\eta_{k+1}\big) + O(n^{L-1}).
\end{align}
Now we prove the two cases separately.

\begin{proof}[Proof of Part (i)]
	
	We choose a constant $c$, depending only on $M$, to be such that
	\begin{enumerate}[label=\textbf{(\alph*)~},leftmargin=*,labelsep=0cm,align=left,itemsep=6pt]
		\item the quantity $\delta_n + O(n^{-1})$ in \eqref{eq:e-diff-new} is bounded above by $\frac{0.23c}{n}$, that is,
		\begin{align}\label{eq:delta-Part-i}
			\big|\delta_n + O(n^{-1})\big| < \frac{0.23c}{n},
		\end{align}
		for every $n\ge 1$ (this is admissible since the big-$O$ term in \eqref{eq:e-diff-new} depends only on $M$ while we have also assumed that $\delta_n = O(n^{-1})$; here note that $L=0$);
		
		\item the inequality
		\begin{align}\label{eq:eta-diff-bound-new}
			\big|\eta_n-\eta_{n-1}\big| < \frac{c}{n}
		\end{align}
		holds for every $n$ with $2\le n\le \max\{200,n_0\}$.
	\end{enumerate}
	
	Inductively, we assume that \eqref{eq:eta-diff-bound-new} holds for each of $2,\ldots,n-1$ with a certain $n>\max\{200,n_0\}$, and we prove \eqref{eq:eta-diff-bound-new} for $n$. Note that this inductive hypothesis implies that
	\begin{align*}
		\left|\sum_{k=1}^{n-2} S_n(k)\cdot  \big(\eta_k-\eta_{k+1}\big)\right| < \sum_{k=1}^{n-2} S_n(k)\cdot \frac{c}{k+1}< \sum_{k=1}^{n-2} S_n(k)\cdot \frac{c}{k}.
	\end{align*}
	Meanwhile, it is known from \eqref{eq:sum-S--1} that
	\begin{align*}
		\sum_{k=1}^{n} S_n(k)\cdot \frac{1}{k} \sim \frac{2}{n}.
	\end{align*}
	Performing some routine computations, we explicitly have, with recourse to \eqref{eq:Snn=1} and \eqref{eq:Sn(n-1)-asymp}, that whenever $n>\max\{200,n_0\}$,
	\begin{align*}
		\sum_{k=1}^{n-2} S_n(k)\cdot \frac{1}{k} < \frac{0.4}{n}.
	\end{align*}
	Since $S_n(n-1)> 0.63$ by \eqref{eq:Sn(n-1)-bound}, we finally deduce from \eqref{eq:e-diff-new}, with \eqref{eq:delta-Part-i} in mind, that
	\begin{align*}
		\big|\eta_{n}-\eta_{n-1}\big| < \left(\frac{0.23c}{n} + \frac{0.4c}{n}\right) \frac{1}{S_n(n-1)}
		<\frac{0.63c}{n}\cdot \frac{1}{0.63}
		=\frac{c}{n}.
	\end{align*}
	
	Now we have proved that \eqref{eq:eta-diff-bound-new} holds for every $n\ge 2$, and it is immediately clear that \eqref{eq:eta-diff-bound} is valid.
\end{proof}

\begin{proof}[Proof of Part (ii)]
	
	We choose a constant $c$, depending only on $L$ and $M$, to be such that
	\begin{enumerate}[label=\textbf{(\alph*)~},leftmargin=*,labelsep=0cm,align=left,itemsep=6pt]
		\item the quantity $\delta_n + O(n^{L-1})$ in \eqref{eq:e-diff-new} is bounded above by $0.26c\cdot n^{L-1}\log n$, that is,
		\begin{align}\label{eq:delta-Part-ii}
			\big|\delta_n + O(n^{L-1})\big| < 0.26c\cdot n^{L-1}\log n,
		\end{align}
		for every $n\ge 1$ (this is admissible since the big-$O$ term in \eqref{eq:e-diff-new} depends only on $M$ while we have also assumed that $\delta_n = O(n^{L-1}\log n)$);
		
		\item the inequality
		\begin{align}\label{eq:eta-diff-bound-2-new}
			\big|\eta_n-\eta_{n-1}\big| < c\cdot n^{L-1}\log n
		\end{align}
		holds for every $n$ with $2\le n\le \max\{200,n_0\}$.
	\end{enumerate}
	
	Inductively, we assume that \eqref{eq:eta-diff-bound-2-new} holds for each of $2,\ldots,n-1$ with a certain $n>\max\{200,n_0\}$, and we prove \eqref{eq:eta-diff-bound-2-new} for $n$. Note that this inductive hypothesis implies that
	\begin{align*}
		\left|\sum_{k=1}^{n-2} S_n(k)\cdot  \big(\eta_k-\eta_{k+1}\big)\right| &< \sum_{k=1}^{n-2} S_n(k)\cdot c\cdot (k+1)^{L+1}\log(k+1)\\
		& \le c\cdot n^{L+1}\log n \cdot \sum_{k=1}^{n-2} S_n(k).
	\end{align*}
	Here we have used the fact that $n^{L-1}\log n$ is non-decreasing with respect to $n$ when $L\ge 1$. Meanwhile, it is known from \eqref{eq:sum-S-0} that
	\begin{align*}
		\sum_{k=1}^{n} S_n(k) < 2.
	\end{align*}
	Recalling that $S_n(n)=1$ from \eqref{eq:Snn=1} and that $S_n(n-1)> 0.63$ from \eqref{eq:Sn(n-1)-bound}, we have
	\begin{align*}
		\sum_{k=1}^{n-2} S_n(k) < 0.37.
	\end{align*}
	Finally, noting the bound given in \eqref{eq:delta-Part-ii}, it follows from \eqref{eq:e-diff-new} that
	\begin{align*}
		\big|\eta_{n}-\eta_{n-1}\big| &< \big(0.26c+0.37c\big)\cdot n^{L-1}\log n \cdot \frac{1}{S_n(n-1)}\\
		&< 0.63c \cdot n^{L-1}\log n \cdot \frac{1}{0.63} \\
		&= c\cdot n^{L-1}\log n.
	\end{align*}
	
	Now we have proved that \eqref{eq:eta-diff-bound-2-new} holds for every $n\ge 2$, and it is immediately clear that \eqref{eq:eta-diff-bound-2} is valid.
\end{proof}

\section{Mean value}\label{sec:mean-value}

In this section, we analyze the behavior of the mean value $\mu_n$ of $X_n$ as $n\rightarrow+\infty$.

\subsection{Explicit formula for the mean value}

Recall that it was shown by Rao et al.~\cite{Rao} that as $n\to+\infty$,
\begin{align*} 
	\mu_n\sim n.
\end{align*}
Alternatively, this asymptotic relation is a direct consequence of Theorem \ref{th:xi-asymp} by choosing $\lambda_n = 1$ for all $n\ge 1$, and then taking $\xi_n=\mu_n$. Now our objective is to elaborate on this asymptotic formula. 

\begin{theorem}\label{th:mu-new-formula}
	We have
	\begin{align}\label{eq:mu-new-formula}
		\mu_n = n + \mathcal{H}_{n-1} + \veps_n,
	\end{align}
	where the limit $\lim_{n\to+\infty} \veps_n$ exists. In particular, for $n\ge 2$,
	\begin{align}\label{eq:epsilon-diff}
		0 < \veps_n-\veps_{n+1} < \frac{1}{n^2}.
	\end{align}
\end{theorem}

We proceed with an argument analogous to the one used in the proof of Theorem \ref{th:xi-asymp} but with more delicacy.

\begin{proof}
	We begin by rewriting the recurrence \eqref{eq:mu-original} as
	\begin{align}\label{eq:mu-Sigma}
		S_n(n-1)\cdot \mu_n = 1+ \Sigma'_1+\Sigma'_2+\Sigma'_3,
	\end{align}
	where
	\begin{align*}
		\Sigma'_1 &:= \sum_{k=1}^{n-1} \frac{A(n,k)}{n!} k,\\
		\Sigma'_2 &:= \sum_{k=1}^{n-1} \frac{A(n,k)}{n!} \mathcal{H}_{k-1},\\
		\Sigma'_3 &:= \sum_{k=1}^{n-1} \frac{A(n,k)}{n!} \veps_k.
	\end{align*}
	Now we evaluate each sum separately by the Abel summation formula.
	
	First, it is clear by \eqref{eq:sum-S-0} that $\Sigma'_1$ equals
	\begin{align}\label{eq:Sigma1-new}
		\Sigma'_1 = \sum_{k=1}^{n-1} \frac{A(n,k)}{n!} k
		= S_n(n-1) \cdot n - \sum_{k=1}^{n-1} S_n(k)
		=S_n(n-1) \cdot n - 1 + \frac{1}{n}.
	\end{align}
	For $\Sigma'_2$, we have, with \eqref{eq:sum-S--1} utilized, that
	\begin{align}\label{eq:Sigma2-new}
		\Sigma'_2 &= \sum_{k=1}^{n-1} \frac{A(n,k)}{n!} \mathcal{H}_{k-1}\notag\\
		&= S_n(n-1) \cdot \mathcal{H}_{n-1} - \sum_{k=1}^{n-1} S_n(k)\cdot \frac{1}{k}\notag\\
		&= S_n(n-1) \cdot \mathcal{H}_{n-1} - \sum_{m=1}^{n-1}\frac{(n-m)!}{m\cdot n!}.
	\end{align}
	Finally, for $\Sigma'_3$, we find that
	\begin{align}\label{eq:Sigma3-new}
		\Sigma'_3 = \sum_{k=1}^{n-1} \frac{A(n,k)}{n!} \veps_k= S_n(n-1) \cdot \veps_{n-1} + \sum_{k=1}^{n-2} S_n(k)\cdot \big(\veps_k-\veps_{k+1}\big).
	\end{align}
	Substituting \eqref{eq:Sigma1-new}, \eqref{eq:Sigma2-new} and \eqref{eq:Sigma3-new} into \eqref{eq:mu-Sigma}, we have
	\begin{align}\label{eq:epsilon-diff-level-2}
		S_n(n-1)\cdot \big(\veps_{n-1}-\veps_n\big) = \sum_{m=2}^{n-1}\frac{(n-m)!}{m\cdot n!} - \sum_{k=1}^{n-2} S_n(k)\cdot \big(\veps_k-\veps_{k+1}\big).
	\end{align}
	
	It is plain that the following asymptotic relations hold:
	\begin{enumerate}[label=\textbf{(\alph*)~},leftmargin=*,labelsep=0cm,align=left,itemsep=6pt]
		\item $\displaystyle S_n(n-1)\sim 1-e^{-1}$;
		
		\item $\displaystyle \sum_{m=2}^{n-1}\frac{(n-m)!}{m\cdot n!}\sim \frac{1}{2n^2}$;
		
		\item $\displaystyle \sum_{t=1}^n S_n(t)\cdot \frac{1}{t^2}\sim \frac{2}{n^2}$.
	\end{enumerate}
	Here \textbf{(a)} is \eqref{eq:Sn(n-1)-asymp}. For \textbf{(b)}, we note that the left-hand side is dominated by the summand with $m=2$. For relation \textbf{(c)}, we shall recall \eqref{eq:sum-S--2}. More precisely, we have explicit bounds: \textit{For $n\ge 200$},
	\begin{enumerate}[label=\textbf{(\alph*')~},leftmargin=*,labelsep=0cm,align=left,itemsep=6pt]
		\item $\displaystyle 0.63 < S_n(n-1)< 0.64$;
  
		\item $\displaystyle \frac{0.49}{n^2} < \sum_{m=2}^{n-1}\frac{(n-m)!}{m\cdot n!}< \frac{0.51}{n^2}$.
	\end{enumerate}
    Here \textbf{(a')} was already given in \eqref{eq:Sn(n-1)-bound}, and \textbf{(b')} can be derived by performing some additional computations for the error terms. Meanwhile, we numerically verify that \eqref{eq:epsilon-diff} is valid whenever $2\le n\le 200$. So throughout we inductively assume \eqref{eq:epsilon-diff} for $2,\ldots, n-1$ with a certain $n> 200$, and we shall prove \eqref{eq:epsilon-diff} for $n$. Noting that $\veps_1-\veps_2 = (-1)-(-1)=0$, we have, with our inductive assumption in mind, that
	\begin{align*}
		0< \sum_{k=1}^{n-2} S_n(k) \big(\veps_k-\veps_{k+1}\big)< \sum_{k=1}^{n-2} S_n(k)\cdot \frac{1}{k^2}.
	\end{align*}
	Recall that $S_n(n)=1$ and that $S_n(n-1)\sim 1-e^{-1}$. Consulting the above asymptotic relation \textbf{(c)}, we carry out some extra computations and find that
	\begin{enumerate}[label=\textbf{(\alph*')~},leftmargin=*,labelsep=0cm,align=left,itemsep=6pt]
		\setcounter{enumi}{2}
		
		\item $\displaystyle 0< \sum_{k=1}^{n-2} S_n(k) \big(\veps_k-\veps_{k+1}\big)<\frac{0.47}{n^2}$.
	\end{enumerate}
	Substituting the above inequalities \textbf{(b')} and \textbf{(c')} into \eqref{eq:epsilon-diff-level-2} gives
	\begin{align*}
		\frac{0.02}{n^2} < S_n(n-1)\cdot \big(\veps_{n-1}-\veps_n\big) < \frac{0.51}{n^2}.
	\end{align*}
	Furthering applying the inequality \textbf{(a')} confirms the required bounds:
	\begin{align*}
		0 < \veps_{n-1}-\veps_n < \frac{1}{(n-1)^2}.
	\end{align*}
	Finally, since the sequence $\{\veps_n\}_{n\ge 2}$ is strictly decreasing but bounded from below by $\veps_1-\zeta(2) = -1 -\zeta(2)$, the \emph{monotone convergence theorem} asserts that it has a limit.
\end{proof}

\subsection{Connection with Bell numbers}\label{sec:Bell}

Our explicit formula for $\mu_n$ allows us to build the following connection between the Bell numbers $B_\ell$ and a certain family of weighted moments of $\mu_n-\mu_k$.

\begin{theorem}\label{th:mu-moments-limit}
	Let $\ell_1$ and $\ell_2$ be nonnegative integers. Then as $n\to +\infty$,
	\begin{align}\label{eq:mu-moments-limit}
		\sum_{k=1}^n \frac{A(n,k)}{n!}\big(n-k\big)^{\ell_1}\big(\mu_n-\mu_k\big)^{\ell_2} = B_{\ell_1+\ell_2} + O(n^{-1}),
	\end{align}
	where the big-$O$ term depends only on $\ell_1$ and $\ell_2$.
\end{theorem}

For this purpose, we first need to bound $\mu_n-\mu_k$.

\begin{lemma}\label{le:mun-muk}
	Let $n\ge 1$. For any $k$ such that $1\le k\le n$,
	\begin{align}\label{eq:mun-muk}
		n-k \le \mu_n-\mu_k \le (n-k)\big(1+\tfrac{1}{k}\big).
	\end{align}
\end{lemma}

\begin{proof}
	Recall from \eqref{eq:mu-new-formula},
    \begin{align*}
		\mu_n-\mu_k = (n-k) + \sum_{j=k}^{n-1}\frac{1}{j} - \sum_{j=k}^{n-1}\big(\veps_j-\veps_{j+1}\big).
	\end{align*}
    It will be sufficient to bound the two summations on the right-hand side of the above. In view of \eqref{eq:epsilon-diff}, it is clear that
	\begin{align*}
		0\le \sum_{j=k}^{n-1}\big(\veps_j-\veps_{j+1}\big)\le \sum_{j=k}^{n-1}\frac{1}{j^2}. 
	\end{align*}
	Thus, 
	\begin{align*}
		\sum_{j=k}^{n-1}\frac{1}{j} - \sum_{j=k}^{n-1}\big(\veps_j-\veps_{j+1}\big) \ge \sum_{j=k}^{n-1}\frac{1}{j} - \sum_{j=k}^{n-1}\frac{1}{j^2}\ge 0.
	\end{align*}
	Meanwhile,
	\begin{align*}
		\sum_{j=k}^{n-1}\frac{1}{j} - \sum_{j=k}^{n-1}\big(\veps_j-\veps_{j+1}\big) \le \sum_{j=k}^{n-1}\frac{1}{j} \le \frac{n-k}{k}.
	\end{align*}
	The claimed inequalities \eqref{eq:mun-muk} then follow.
\end{proof}

Now we complete the proof of Theorem \ref{th:mu-moments-limit}.

\begin{proof}[Proof of Theorem \ref{th:mu-moments-limit}]
	By virtue of Lemma \ref{le:mun-muk}, we know that for $n\ge 1$,
	\begin{align*}
		\sum_{k=1}^n \frac{A(n,k)}{n!}\big(n-k\big)^{\ell_1}\big(\mu_n-\mu_k\big)^{\ell_2}\ge \sum_{k=1}^n \frac{A(n,k)}{n!}\big(n-k\big)^{\ell_1+\ell_2}
	\end{align*}
	and
	\begin{align*}
		\sum_{k=1}^n \frac{A(n,k)}{n!}\big(n-k\big)^{\ell_1}\big(\mu_n-\mu_k\big)^{\ell_2}\le \sum_{k=1}^n \frac{A(n,k)}{n!}\big(n-k\big)^{\ell_1+\ell_2}\big(1+\tfrac{1}{k}\big)^{\ell_2}.
	\end{align*}
	Recalling from \eqref{eq:A-(n-k)u},
	\begin{align}\label{eq:limit-1}
		\sum_{k=1}^n \frac{A(n,k)}{n!}(n-k)^{\ell_1+\ell_2} = B_{\ell_1+\ell_2} + O(n^{-1}),
	\end{align}
	it is sufficient to show that we also have
	\begin{align}\label{eq:limit-1+1/k}
		\sum_{k=1}^n \frac{A(n,k)}{n!}\big(n-k\big)^{\ell_1+\ell_2}\big(1+\tfrac{1}{k}\big)^{\ell_2} = B_{\ell_1+\ell_2} + O(n^{-1}).
	\end{align}
	To see this, we expand $\big(1+\frac{1}{k}\big)^{\ell_2}$ as $1+\sum_{s=1}^{\ell_2} \binom{\ell_2}{s}\frac{1}{k^s}$, and apply \eqref{eq:Q-s-bound} to get
	\begin{align*}
		\sum_{s=1}^{\ell_2} \binom{\ell_2}{s}\sum_{k=1}^n \frac{A(n,k)}{n!}(n-k)^{\ell_1+\ell_2} \frac{1}{k^s} = O(n^{-1}).
	\end{align*}
	Thus,
	\begin{align*}
		\sum_{k=1}^n \frac{A(n,k)}{n!}\big(n-k\big)^{\ell_1+\ell_2}\big(1+\tfrac{1}{k}\big)^{\ell_2} = \sum_{k=1}^n \frac{A(n,k)}{n!}(n-k)^{\ell_1+\ell_2} + O(n^{-1}),
	\end{align*}
	which gives \eqref{eq:limit-1+1/k} by recalling \eqref{eq:limit-1}.
\end{proof}


\begin{example}
	We have
	\begin{align}
		\sum_{k=1}^n \frac{A(n,k)}{n!}\big(\mu_n-\mu_k\big) &= 1 + O(n^{-1}),\label{eq:mu-1st-moments-limit}\\
		\sum_{k=1}^n \frac{A(n,k)}{n!}\big(\mu_n-\mu_k\big)^2 &= 2 + O(n^{-1}),\\
		\sum_{k=1}^n \frac{A(n,k)}{n!}\big(\mu_n-\mu_k\big)^3 &= 5 + O(n^{-1}).
	\end{align}
\end{example}

\begin{remark}
We may further elaborate on \eqref{eq:mu-1st-moments-limit}. That is, by virtue of  \eqref{eq:mu-original},  it is immediately clear that for $n\geq 2$, 
	\begin{align}\label{eq:mu-1st-moments-limit-new}
		\sum_{k=1}^n \frac{A(n,k)}{n!}\big(\mu_n-\mu_k\big) = 1.
	\end{align}
\end{remark}


The argument for Theorem \ref{th:mu-moments-limit} also works, mutatis mutandis, for the following estimate by a direct application of \eqref{eq:Q-s-bound}.

\begin{theorem}
	Let $\ell_1$ and $\ell_2$ be nonnegative integers. Then as $n\to +\infty$,
	\begin{align}\label{eq:mu-moments-limit-k-1}
		\sum_{k=1}^n \frac{A(n,k)}{n!}\big(n-k\big)^{\ell_1}\big(\mu_n-\mu_k\big)^{\ell_2}\frac{1}{k} = O(n^{-1}).
	\end{align}
	Here the big-$O$ term depends only on $\ell_1$ and $\ell_2$.
\end{theorem}

Finally, we derive the following result from Theorem \ref{th:mu-moments-limit}.

\begin{corollary}\label{coro:mu-moments-limit-new}
	Let $\ell_1$ and $\ell_2$ be nonnegative integers. Then as $n\to +\infty$,
	\begin{align}\label{eq:mu-moments-limit-new}
		\sum_{k=1}^n \frac{A(n,k)}{n!}k^{\ell_1}\big(\mu_n-\mu_k\big)^{\ell_2} = B_{\ell_2} n^{\ell_1} + O(n^{\ell_1-1}).
	\end{align}
	Here the big-$O$ term depends only on $\ell_1$ and $\ell_2$.
\end{corollary}

\begin{proof}
	When $\ell_1=0$, the result is exactly a specialization of Theorem \ref{th:mu-moments-limit}. In what follows, we assume that $\ell_1\ge 1$ and expand $k^{\ell_1}$ as
	\begin{align}\label{eq:expansion-k}
		k^{\ell_1} = \big(n-(n-k)\big)^{\ell_1} = n^{\ell_1} + \sum_{s=1}^{\ell_1} \binom{\ell_1}{s} (-1)^s \big(n-k\big)^s n^{\ell_1-s}.
	\end{align}
	In light of Theorem \ref{th:mu-moments-limit},
	\begin{align*}
		n^{\ell_1}\cdot \sum_{k=1}^n \frac{A(n,k)}{n!}\big(\mu_n-\mu_k\big)^{\ell_2} = B_{\ell_2} n^{\ell_1} + O(n^{\ell_1-1}).
	\end{align*}
	Also, for each $s$ with $1\le s\le \ell_1$, 
	\begin{align*}
		n^{\ell_1-s}\cdot \sum_{k=1}^n \frac{A(n,k)}{n!}\binom{\ell_1}{s} (-1)^s \big(n-k\big)^s\big(\mu_n-\mu_k\big)^{\ell_2} = O(n^{\ell_1-s}).
	\end{align*}
	The desired result then follows.
\end{proof}

\subsection{Other estimates}

Herein we establish a few more estimates for later use. All asymptotic relations in this subsection depend only on $L$ and $M$.

\begin{lemma}\label{le:mu-moments-limit-LM}
	Let $L\ge 0$ and $M\ge 1$ be integers. Then as $n\to +\infty$,
	\begin{align}\label{eq:mu-moments-limit-LM}
		\sum_{k=1}^n \frac{A(n,k)}{n!}k^{L}\big(1+\mu_k-\mu_n\big)^{M} = \left(\sum_{m=0}^M \binom{M}{m}(-1)^m B_m\right) n^{L} + O(n^{L-1}).
	\end{align}
\end{lemma}

\begin{proof}
	We make use of the expansion
	\begin{align}\label{eq:expansion-mu-k-n}
		\big(1+\mu_k-\mu_n\big)^{M} = \sum_{m=0}^M \binom{M}{m}(-1)^m \big(\mu_n-\mu_k\big)^{m},
	\end{align}
	and then apply \eqref{eq:mu-moments-limit-new}.
\end{proof}

\begin{example}
	Recall that
	\begin{align*}
		B_0=1,\qquad B_1=1,\qquad B_2=2,\qquad B_3=5.
	\end{align*}
	The following special cases will be used in the sequel:
	\begin{align}
		\sum_{k=1}^n \frac{A(n,k)}{n!}k^{L}\big(1+\mu_k-\mu_n\big)^{2} &= n^{L} + O(n^{L-1}),\label{eq:mu-moments-limit-L2}\\
		\sum_{k=1}^n \frac{A(n,k)}{n!}k^{L}\big(1+\mu_k-\mu_n\big)^{3} &= - n^{L} + O(n^{L-1}).\label{eq:mu-moments-limit-L3}
	\end{align}
\end{example}

A closer look at the $M = 1$ case of \eqref{eq:mu-moments-limit-LM} reveals that
\begin{align*}
	\sum_{m=0}^1 \binom{1}{m}(-1)^m B_m = B_0 - B_1 = 1 - 1 = 0,
\end{align*}
thereby indicating that the main term in \eqref{eq:mu-moments-limit-LM} vanishes. Hence, we make an elaboration.

\begin{lemma}\label{le:mu-moments-limit-L1}
	Let $L\ge 0$ be an integer. Then as $n\to +\infty$,
	\begin{align}\label{eq:mu-moments-limit-L1}
		\sum_{k=1}^n \frac{A(n,k)}{n!}k^{L}\big(1+\mu_k-\mu_n\big) = L n^{L-1} + O(n^{L-2}).
	\end{align}
\end{lemma}

\begin{proof}
	For the case when $L=0$, we note from \eqref{sum=1} that
	\begin{equation*}
		\sum_{k=1}^n \frac{A(n,k)}{n!} =1,
	\end{equation*}
	and from \eqref{eq:mu-1st-moments-limit-new} that
	\begin{equation*}
		\sum_{k=1}^n \frac{A(n,k)}{n!}\big(\mu_n-\mu_k\big) = 1.
	\end{equation*}
	Thus,
	\begin{align}\label{eq:mu-moments-limit-01}
		\sum_{k=1}^n \frac{A(n,k)}{n!}\big(1+\mu_k-\mu_n\big) = 0.
	\end{align}
	
	Now assume that $L\ge 1$. We again use \eqref{eq:expansion-k} but with one more term taken out:
	\begin{align}\label{eq:expansion-k-new}
		k^L = n^L - L n^{L-1} \big(n-k\big) + \sum_{\ell=2}^{L} \binom{L}{\ell} (-1)^\ell \big(n-k\big)^\ell n^{L-\ell}.
	\end{align}
	As pointed out earlier, the contribution from the $n^L$ term vanishes in \eqref{eq:mu-moments-limit-L1}. Thus,
	\begin{align*}
		\sum_{k=1}^n \frac{A(n,k)}{n!}k^{L}\big(1+\mu_k-\mu_n\big) &= - L n^{L-1}\sum_{k=1}^n \big(n-k\big)\big(1+\mu_k-\mu_n\big) +O(n^{L-2})\\
		&= - L n^{L-1} \big(B_1-B_2+O(n^{-1})\big) +O(n^{L-2})\\
		&= L n^{L-1} + O(n^{L-2}),
	\end{align*}
	as claimed.
\end{proof}

Finally, we have the following result.

\begin{lemma}\label{le:mu-moments-limit-LM-log}
	Let $M\ge 1$ be an integer. For any sequence $f(n)=O(n^L \log n)$ with $L\ge 0$ an integer, we have, as $n\to +\infty$,
	\begin{align}\label{eq:mu-moments-limit-LM-log}
		\sum_{k=1}^n \frac{A(n,k)}{n!}f(k) \big(1+\mu_k-\mu_n\big)^{M} = O(n^{L}\log n).
	\end{align}
\end{lemma}

\begin{proof}
    It is clear that
    \begin{align*}
        \sum_{k=1}^n \frac{A(n,k)}{n!}f(k) \big(1+\mu_k-\mu_n\big)^{M} = O\left(n^{L}\log n \sum_{m=0}^M \binom{M}{m} \sum_{k=1}^n \frac{A(n,k)}{n!} \big(\mu_n-\mu_k\big)^{m}\right).
    \end{align*}
    Noting that the inner summation on the right-hand side is $O(1)$ by \eqref{eq:mu-moments-limit}, we arrive at the desired relation.
\end{proof}

\section{Proof of Theorem \ref{th:c-mom}}\label{sec:var-3moment}

Now, we are ready to prove Theorem \ref{th:c-mom}. To do so, we require the following elaboration.

\begin{theorem}\label{th:c-mom-new}
	For every $m\ge 2$, Theorem \ref{th:c-mom} is valid. Furthermore, if we define
	\begin{align*}
		\varepsilon_n^{(m)} := \bbE\big[(X_n-\mu_n)^m\big]- \begin{cases}
			(2M-1)!!\cdot n^M, & \text{if $m=2M$},\\
			\tfrac{2}{3}M(2M+1)!!\cdot n^M, & \text{if $m=2M+1$},
		\end{cases}
	\end{align*}
	then
	\begin{align}\label{eq:mom-error-diff}
		\big|\veps^{(m)}_n-\veps^{(m)}_{n-1}\big| = \begin{cases}
			O(n^{-1}), & \text{if $m=2$ or $3$},\\
			O(n^{\lfloor\frac{m}{2}\rfloor-2}\log n), & \text{if $m\ge 4$}.
		\end{cases}
	\end{align}
\end{theorem}

\begin{proof}[Outline of the Proof]
    We begin with \eqref{eq:E-poly} and find that
\begin{align*}
	\bbE\left[(X_n-\mu_n)^m\right] & =\sum_{k=1}^{n}\frac{A(n,k)}{n!}\cdot \bbE\left[\big(1+X_{k}-\mu_{n}\big)^{m}\right]\\
	&=\sum_{k=1}^{n}\frac{A(n,k)}{n!}\cdot \bbE\left[\big((X_{k}-\mu_{k})+(1+\mu_{k}-\mu_{n})\big)^{m}\right].
\end{align*}
Hence,
\begin{align}\label{eq:mth-M}
	\left(1-\frac{A(n,n)}{n!}\right)\bbE\big[(X_n-\mu_n)^m\big] = \sum_{k=1}^{n-1}\frac{A(n,k)}{n!}\,\bbE\big[(X_k-\mu_k)^m\big] + \sum_{\ell=1}^m \binom{m}{\ell} I_\ell^{(m)},
\end{align}
where for each $\ell$ with $1\le \ell\le m$,
\begin{align*}
	I_\ell^{(m)} := \sum_{k=1}^{n}\frac{A(n,k)}{n!}\cdot \bbE\big[(X_k-\mu_k)^{m-\ell}\big] \big(1+\mu_k-\mu_n\big)^{\ell}.
\end{align*}

In what follows, we first work on the variance and the third central moment in Subsections~\ref{sec:var} and \ref{sec:3rd}, respectively. Using them as initial cases, we then evaluate the higher central moments by an inductive argument in Subsection~\ref{sec:higher}. In this course, we shall consider even- and odd-order moments separately.
\end{proof}


\subsection{Asymptotic relation for the variance}\label{sec:var}

For the variance $\Var[X_n]=\bbE\big[(X_n-\mu_n)^2\big]$, we explicitly paste \eqref{eq:mth-M}:
\begin{align}\label{eq:Var-rec}
	\left(1-\frac{A(n,n)}{n!}\right)\Var[X_n] = \sum_{k=1}^{n-1}\frac{A(n,k)}{n!}\,\Var[X_k] + I_1^{(2)} + I_2^{(2)}, 
\end{align}
where
\begin{align*}
	I_1^{(2)} &= \sum_{k=1}^{n}\frac{A(n,k)}{n!}\cdot \bbE\big[X_k-\mu_k\big] \big(1+\mu_k-\mu_n\big),\\
	I_2^{(2)} &= \sum_{k=1}^{n}\frac{A(n,k)}{n!}\cdot \big(1+\mu_{k}-\mu_{n}\big)^2.
\end{align*}

In light of \eqref{eq:mu-moments-limit-L2},
\begin{align}\label{eq:I2-2}
	I_2^{(2)} = 1 + O(n^{-1}).
\end{align}
Also, using the fact that
\begin{align}\label{eq:E-diff}
	\bbE\left[X_{k}-\mu_{k}\right] = \bbE\left[X_{k}\right]-\mu_{k} = \mu_k-\mu_k = 0,
\end{align}
we have
\begin{align}\label{eq:I2-1}
	I_1^{(2)} = 0.
\end{align}

Consequently, we insert \eqref{eq:I2-2} and \eqref{eq:I2-1} into \eqref{eq:Var-rec}, and obtain
\begin{align*}
	\left(1-\frac{A(n,n)}{n!}\right)\Var[X_n] = \sum_{k=1}^{n-1}\frac{A(n,k)}{n!}\,\Var[X_k] + 1 + O(n^{-1}).
\end{align*}
Applying Theorems \ref{th:xi-asymp} and \ref{th:xi-asymp-error} yields the desired result.

\subsection{Asymptotic relation for the third central moment}\label{sec:3rd}

The third central moment $\bbE\big[(X_n-\mu_n)^3\big]$ satisfies:
\begin{align}\label{eq:3rdM}
	\left(1-\frac{A(n,n)}{n!}\right)\bbE\big[(X_n-\mu_n)^3\big] = \sum_{k=1}^{n-1}\frac{A(n,k)}{n!}\,\bbE\big[(X_k-\mu_k)^3\big] + 3I_1^{(3)} + 3I_2^{(3)} + I_3^{(3)},
\end{align}
where
\begin{align*}
	I_1^{(3)} &= \sum_{k=1}^{n}\frac{A(n,k)}{n!}\cdot \bbE\big[(X_k-\mu_k)^{2}\big] \big(1+\mu_k-\mu_n\big),\\
	I_2^{(3)} &= \sum_{k=1}^{n}\frac{A(n,k)}{n!}\cdot \bbE\big[X_k-\mu_k\big] \big(1+\mu_{k}-\mu_{n}\big)^2,\\
	I_3^{(3)} &= \sum_{k=1}^{n}\frac{A(n,k)}{n!}\cdot \big(1+\mu_{k}-\mu_{n}\big)^3.
\end{align*}

First, by \eqref{eq:mu-moments-limit-L3},
\begin{align}\label{eq:I3-3}
	I_3^{(3)} = -1 + O(n^{-1}).
\end{align}
Also, \eqref{eq:E-diff} tells us that
\begin{align}\label{eq:I3-2}
	I_2^{(3)} = 0.
\end{align}
Now it remains to evaluate $I_1^{(3)}$. Recalling Theorem \ref{th:c-mom-new} for $m=2$, we split $I_1^{(3)}$ as
\begin{align*}
	I_1^{(3)} = I_{1,1}^{(3)} + I_{1,2}^{(3)},
\end{align*}
where
\begin{align*}
	I_{1,1}^{(3)} &:= \sum_{k=1}^{n}\frac{A(n,k)}{n!}\cdot k \big(1+\mu_k-\mu_n\big),\\
	I_{1,2}^{(3)} &:= \sum_{k=1}^{n}\frac{A(n,k)}{n!}\cdot \varepsilon^{(2)}_k \big(1+\mu_k-\mu_n\big).
\end{align*}
In view of \eqref{eq:mu-moments-limit-L1}, it is clear that
\begin{align}\label{eq:I3-11}
	I_{1,1}^{(3)} = 1 + O(n^{-1}).
\end{align}
In what follows, our objective is to show
\begin{equation}\label{eq:I3-12}
	I_{1,2}^{(3)} = O(n^{-1}).
\end{equation}
To begin with, we define, for $1\le t\le n$,
\begin{align}\label{eq:R-def}
	R_n(t):= \sum_{k=1}^t \frac{A(n,k)}{n!}\big(1+\mu_k-\mu_n\big).
\end{align}
Clearly, $R_n(t)\le 0$ whenever $1\le t\le n-1$ since $\mu_k-\mu_n\le -1$ for every $k$ with $1\le k\le t\le n-1$ according to Lemma \ref{le:mun-muk}. In the meantime, it was shown in \eqref{eq:mu-moments-limit-01} that
\begin{align*}
	R_n(n)=0.
\end{align*}
In light of the Abel summation formula,
\begin{align*}
	I_{1,2}^{(3)} = R_n(n) \veps^{(2)}_{n+1} + \sum_{k=1}^{n} R_n(k) \big(\veps^{(2)}_{k}-\veps^{(2)}_{k+1}\big)= \sum_{k=1}^{n} R_n(k) \big(\veps^{(2)}_{k}-\veps^{(2)}_{k+1}\big).
\end{align*}
Recalling from \eqref{eq:mom-error-diff}, there exists a constant $c^{(2)}$ such that
\begin{align*}
	\big|\veps^{(2)}_{k}-\veps^{(2)}_{k+1}\big|\le \frac{c^{(2)}}{k}
\end{align*}
for every $k$ with $1\le k\le n$. Thus,
\begin{align*}
	\big|I_{1,2}^{(3)}\big|\le -\sum_{k=1}^{n} R_n(k) \big|\veps^{(2)}_{k}-\veps^{(2)}_{k+1}\big|\le -\sum_{k=1}^{n} R_n(k)\cdot \frac{c^{(2)}}{k}.
\end{align*}
Let
\begin{align*}
	\widehat{I}_{1,2}^{(3)}:= \sum_{k=1}^{n} R_n(k)\cdot \frac{1}{k}.
\end{align*}
To prove \eqref{eq:I3-12}, it will be sufficient to show that
\begin{equation} \label{eq:I3-22-new}
	\widehat{I}_{1,2}^{(3)} = O(n^{-1}).
\end{equation}
Utilizing the fact that $R_n(n)=0$ and recalling the definition \eqref{eq:R-def}, we have
\begin{align*}
	\widehat{I}_{1,2}^{(3)} = \sum_{k=1}^{n-1} \frac{1}{k} \sum_{j=1}^k \frac{A(n,j)}{n!}\big(1+\mu_j-\mu_n\big)
	= \sum_{j=1}^{n-1} \frac{A(n,j)}{n!}\big(1+\mu_j-\mu_n\big)\cdot \sum_{k=j}^{n-1} \frac{1}{k}.
\end{align*}
Thus, noting the plain inequality
\begin{align*}
	\sum_{k=j}^{n-1} \frac{1}{k}\le \frac{n-j}{j},
\end{align*}
we have
\begin{align*}
	\big|\widehat{I}_{1,2}^{(3)}\big| &\le \sum_{j=1}^{n-1} \frac{A(n,j)}{n!}\cdot \sum_{k=j}^{n-1} \frac{1}{k} + \sum_{j=1}^{n-1} \frac{A(n,j)}{n!}\big(\mu_n-\mu_j\big) \cdot \sum_{k=j}^{n-1} \frac{1}{k}\\
	&\le \sum_{j=1}^{n-1} \frac{A(n,j)}{n!} \frac{n-j}{j} + \sum_{j=1}^{n-1} \frac{A(n,j)}{n!}\big(\mu_n-\mu_j\big)\frac{n-j}{j}\\
	&= O(n^{-1}),
\end{align*}
as claimed in \eqref{eq:I3-22-new}. Here we have applied \eqref{eq:mu-moments-limit-k-1} to each summation in the second line for the final asymptotic relation. Hence, combining \eqref{eq:I3-11} and \eqref{eq:I3-12} gives us
\begin{align}\label{eq:I3-1}
	I_{1}^{(3)} = 1 + O(n^{-1}).
\end{align}

By substituting \eqref{eq:I3-3}, \eqref{eq:I3-2} and \eqref{eq:I3-1} into \eqref{eq:3rdM}, we have
\begin{align*}
	\left(1-\frac{A(n,n)}{n!}\right)\bbE\big[(X_n-\mu_n)^3\big] = \sum_{k=1}^{n-1}\frac{A(n,k)}{n!}\,\bbE\big[(X_k-\mu_k)^3\big] + 2 + O(n^{-1}).
\end{align*}
The required result follows from Theorems \ref{th:xi-asymp} and \ref{th:xi-asymp-error}.

\subsection{Asymptotic relations for higher central moments}\label{sec:higher}

For higher central moments in \eqref{eq:mom-asymp}, we shall use mathematical induction by first assuming the validity for $2,3,\ldots, 2M-1$, and then proving the cases of $2M$ and $2M+1$ sequentially.

\subsubsection{Even-order central moments}

Recall from \eqref{eq:mth-M} that
\begin{align*}
	\left(1-\frac{A(n,n)}{n!}\right)\bbE\big[(X_n-\mu_n)^{2M}\big] = \sum_{k=1}^{n-1}\frac{A(n,k)}{n!}\,\bbE\big[(X_k-\mu_k)^{2M}\big]
	+ \sum_{\ell=1}^{2M} \binom{2M}{\ell} I_\ell^{(2M)},
\end{align*}
where
\begin{align*}
	I_1^{(2M)} &= \sum_{k=1}^{n}\frac{A(n,k)}{n!}\cdot \bbE\big[(X_k-\mu_k)^{2M-1}\big] \big(1+\mu_k-\mu_n\big),\\
	I_2^{(2M)} &= \sum_{k=1}^{n}\frac{A(n,k)}{n!}\cdot \bbE\big[(X_k-\mu_k)^{2M-2}\big] \big(1+\mu_k-\mu_n\big)^{2},\\
	&=\quad \vdots \\
	I_{2M-1}^{(2M)} &= \sum_{k=1}^{n}\frac{A(n,k)}{n!}\cdot \bbE\big[X_k-\mu_k\big] \big(1+\mu_{k}-\mu_{n}\big)^{2M-1},\\
	I_{2M}^{(2M)} &= \sum_{k=1}^{n}\frac{A(n,k)}{n!}\cdot \big(1+\mu_{k}-\mu_{n}\big)^{2M}.
\end{align*}

As before,
\begin{align*}
	I_{2M}^{(2M)} = O(1),
\end{align*}
and
\begin{align*}
	I_{2M-1}^{(2M)} = 0.
\end{align*}
Also, for every $L$ with $1\le L\le M-1$, we know from the inductive hypothesis that
\begin{align*}
	\bbE\big[(X_n-\mu_n)^{2L}\big] &= (2L-1)!!\cdot n^L + O(n^{L-1}\log n),\\
	\bbE\big[(X_n-\mu_n)^{2L+1}\big] &= \tfrac{2}{3}L(2L+1)!!\cdot n^L + O(n^{L-1}\log n).
\end{align*}
In light of \eqref{eq:mu-moments-limit-LM}, \eqref{eq:mu-moments-limit-L1} and \eqref{eq:mu-moments-limit-LM-log}, we derive that for $2\le L'\le M-1$,
\begin{align*}
	I_{2L'-1}^{(2M)} &= O(n^{M-L'}),\\
	I_{2L'}^{(2M)} &= O(n^{M-L'}).
\end{align*}
Also,
\begin{align*}
	I_{1}^{(2M)} &= O(n^{M-2}\log n),\\
	I_{2}^{(2M)} &= (2M-3)!!\cdot n^{M-1} + O(n^{M-2}\log n).
\end{align*}

Therefore,
\begin{align*}
	\left(1-\frac{A(n,n)}{n!}\right)\bbE\big[(X_n-\mu_n)^{2M}\big] &= \sum_{k=1}^{n-1}\frac{A(n,k)}{n!}\,\bbE\big[(X_k-\mu_k)^{2M}\big]\\
	&\quad + \tbinom{2M}{2}(2M-3)!!\cdot n^{M-1} + O(n^{M-2}\log n).
\end{align*}
Noting that
\begin{align*}
	\tfrac{1}{(M-1)+1}\cdot \tbinom{2M}{2}(2M-3)!! = (2M-1)!!,
\end{align*}
we conclude the case of $2M$ in Theorem \ref{th:c-mom-new} by virtue of Theorems \ref{th:xi-asymp} and \ref{th:xi-asymp-error}.

\subsubsection{Odd-order central moments}

By virtue of \eqref{eq:mth-M},
\begin{align*}
	\left(1-\frac{A(n,n)}{n!}\right)\bbE\big[(X_n-\mu_n)^{2M+1}\big] &= \sum_{k=1}^{n-1}\frac{A(n,k)}{n!}\,\bbE\big[(X_k-\mu_k)^{2M+1}\big]\\
    &\quad + \sum_{\ell=1}^{2M+1} \binom{2M+1}{\ell} I_\ell^{(2M+1)},
\end{align*}
where
\begin{align*}
	I_1^{(2M+1)} &= \sum_{k=1}^{n}\frac{A(n,k)}{n!}\cdot \bbE\big[(X_k-\mu_k)^{2M}\big] \big(1+\mu_k-\mu_n\big),\\
	I_2^{(2M+1)} &= \sum_{k=1}^{n}\frac{A(n,k)}{n!}\cdot \bbE\big[(X_k-\mu_k)^{2M-1}\big] \big(1+\mu_k-\mu_n\big)^{2},\\
	I_3^{(2M+1)} &= \sum_{k=1}^{n}\frac{A(n,k)}{n!}\cdot \bbE\big[(X_k-\mu_k)^{2M-2}\big] \big(1+\mu_k-\mu_n\big)^{3},\\
	&=\quad \vdots \\
	I_{2M}^{(2M+1)} &= \sum_{k=1}^{n}\frac{A(n,k)}{n!}\cdot \bbE\big[X_k-\mu_k\big] \big(1+\mu_{k}-\mu_{n}\big)^{2M},\\
	I_{2M+1}^{(2M+1)} &= \sum_{k=1}^{n}\frac{A(n,k)}{n!}\cdot \big(1+\mu_{k}-\mu_{n}\big)^{2M+1}.
\end{align*}

In the same vein,
\begin{align*}
	I_{2M+1}^{(2M+1)} = O(1),
\end{align*}
and
\begin{align*}
	I_{2M}^{(2M+1)} = 0.
\end{align*}
Also, for $2\le L'\le M-1$,
\begin{align*}
	I_{2L'}^{(2M+1)} &= O(n^{M-L'}),\\
	I_{2L'+1}^{(2M+1)} &= O(n^{M-L'}).
\end{align*}
Furthermore,
\begin{align*}
	I_{2}^{(2M+1)} &= \tfrac{2}{3}(M-1)(2M-1)!!\cdot n^{M-1} + O(n^{M-2}\log n),\\
	I_{3}^{(2M+1)} &= -(2M-3)!!\cdot n^{M-1} + O(n^{M-2}\log n).
\end{align*}
It remains to estimate $I_{1}^{(2M+1)}$. Recalling from the inductive hypothesis that
\begin{align*}
	\bbE\big[(X_k-\mu_k)^{2M}\big] = (2M-1)!!\cdot n^{M} + \varepsilon^{(2M)}_n,
\end{align*}
we rewrite $I_{1}^{(2M+1)}$ as
\begin{align}\label{eq:I-2m+1-1-def}
	I_{1}^{(2M+1)} = (2M-1)!!\cdot I_{1,1}^{(2M+1)} + I_{1,2}^{(2M+1)}
\end{align}
where
\begin{align*}
	I_{1,1}^{(2M+1)}&:=\sum_{k=1}^{n}\frac{A(n,k)}{n!}\cdot k^{M} \big(1+\mu_k-\mu_n\big),\\
	I_{1,2}^{(2M+1)}&:=\sum_{k=1}^{n}\frac{A(n,k)}{n!}\cdot \varepsilon^{(2M)}_k \big(1+\mu_k-\mu_n\big).
\end{align*}
Now we first derive from \eqref{eq:mu-moments-limit-L1} that
\begin{align}\label{eq:I2M+1-11}
	I_{1,1}^{(2M+1)} = M\cdot n^{M-1} + O(n^{M-2}).
\end{align}
Next, we claim that
\begin{equation}\label{eq:I2M+1-12}
	I_{1,2}^{(2M+1)} = O(n^{M-2}\log n).
\end{equation}
Recalling the partial sums defined in \eqref{eq:R-def}, we apply the Abel summation formula and obtain
\begin{align*}
	I_{1,2}^{(2M+1)} = \sum_{k=1}^{n} R_n(k) \big(\veps^{(2M)}_{k}-\veps^{(2M)}_{k+1}\big).
\end{align*}
Recalling from \eqref{eq:mom-error-diff}, there exists a constant $c^{(2M)}$ such that
\begin{align*}
	\big|\veps^{(2M)}_{k}-\veps^{(2M)}_{k+1}\big|\le c^{(2M)}k^{M-2}\log k
\end{align*}
for every $k$ with $1\le k\le n$. Thus,
\begin{align*}
	\big|I_{1,2}^{(2M+1)}\big|\le -\sum_{k=1}^{n} R_n(k) \big|\veps^{(2M)}_{k}-\veps^{(2M)}_{k+1}\big|\le -\sum_{k=1}^{n} R_n(k)\cdot c^{(2M)}k^{M-2}\log k.
\end{align*}
Let
\begin{align*}
	\widehat{I}_{1,2}^{(2M+1)}:= \sum_{k=1}^{n} R_n(k)\cdot k^{M-2}\log k.
\end{align*}
We shall prove that
\begin{equation} \label{eq:I2M+1-12-new}
	\widehat{I}_{1,2}^{(2M+1)} = O(n^{M-2}\log n),
\end{equation}
thereby yielding \eqref{eq:I2M+1-12}. Since $R_n(n)=0$, we recall the definition \eqref{eq:R-def} and get
\begin{align*}
	\widehat{I}_{1,2}^{(2M+1)} &= \sum_{k=1}^{n-1} k^{M-2}\log k \sum_{j=1}^k \frac{A(n,j)}{n!}\big(1+\mu_j-\mu_n\big)\\
	&= \sum_{j=1}^{n-1} \frac{A(n,j)}{n!}\cdot \sum_{k=j}^{n-1} k^{M-2}\log k - \sum_{j=1}^{n-1} \frac{A(n,j)}{n!}\big(\mu_n-\mu_j\big)\cdot \sum_{k=j}^{n-1} k^{M-2}\log k.
\end{align*}
Noting that
\begin{align*}
	\sum_{k=j}^{n-1} k^{M-2}\log k\le n^{M-2}\log n\cdot \big(n-j\big),
\end{align*}
we have, by virtue of \eqref{eq:mu-moments-limit},
\begin{align*}
	\big|\widehat{I}_{1,2}^{(2M+1)}\big| &\le n^{M-2}\log n\cdot \left(\sum_{j=1}^{n-1} \frac{A(n,j)}{n!} \big(n-j\big) + \sum_{j=1}^{n-1} \frac{A(n,j)}{n!}\big(n-j\big)\big(\mu_n-\mu_j\big)\right)\\
	&= O(n^{M-2}\log n).
\end{align*}
Therefore, \eqref{eq:I2M+1-12-new} holds, so does \eqref{eq:I2M+1-12}. Inserting \eqref{eq:I2M+1-11} and \eqref{eq:I2M+1-12} into \eqref{eq:I-2m+1-1-def} gives
\begin{align*}
	I_{1}^{(2M+1)} = M(2M-1)!!\cdot n^{M-1} + O(n^{M-2}\log n).
\end{align*}

In conclusion,
\begin{align*}
	&\left(1-\frac{A(n,n)}{n!}\right)\bbE\big[(X_n-\mu_n)^{2M+1}\big]\\
	&= \sum_{k=1}^{n-1}\frac{A(n,k)}{n!}\,\bbE\big[(X_k-\mu_k)^{2M+1}\big]\\
	&\quad + \left(\tbinom{2M+1}{1}M(2M-1)!!+\tbinom{2M+1}{2}\tfrac{2}{3}(M-1)(2M-1)!!-\tbinom{2M+1}{3}(2M-3)!!\right)\cdot n^{M-1}\\
	&\quad + O(n^{M-2}\log n).
\end{align*}
Noting that $\tfrac{2}{3}M(2M+1)!!$ equals
\begin{align*}
	\tfrac{1}{(M-1)+1}\cdot \left(\tbinom{2M+1}{1} M(2M-1)!!+\tbinom{2M+1}{2}\tfrac{2}{3}(M-1)(2M-1)!!-\tbinom{2M+1}{3}(2M-3)!!\right),
\end{align*}
we arrive at the case of $2M+1$ in Theorem \ref{th:c-mom-new} after applying Theorems \ref{th:xi-asymp} and \ref{th:xi-asymp-error}.

\section*{Acknowledgment}

The second and last authors would like to thank our colleague, Dr. Xingshi Cai from Duke Kunshan University, for the discussion at the early stage of this project, in our Discrete Math Seminar.

\bibliographystyle{amsplain}

\end{document}